\documentclass[11pt]{amsart}
\usepackage{latexsym}
\usepackage{palatino}
\usepackage{amsmath, amssymb}
\usepackage[dvipsnames,usenames]{color}
\usepackage[all]{xy}
\newtheorem{dfs}{Definition}[section]
\newtheorem{lms}[dfs]{Lemma}
\newtheorem{thms}[dfs]{Theorem}

\newtheorem{cors}[dfs]{Corollary}
\newtheorem{rems}[dfs]{Remark}

\newtheorem*{thm*}{Theorem}

\title[Ranks of operators in simple C$^*$-algebras]
{Ranks of operators in simple C$^*$-algebras}

\author{Marius Dadarlat}
\address{Department of Mathematics, Purdue University,
150 N.~University St., West Lafayette, IN, 47907-2067, U.S.A.}
\email{mdd@math.purdue.edu}

\author{Andrew S.\  Toms}
\address{Department of Mathematics and Statistics, York University,
4700 Keele St., Toronto, Ontario, Canada, M3J 1P3}
\email{atoms@mathstat.yorku.ca}
\date{\today}
\thanks{M.D. was partially supported by NSF grant \#DMS-0801173; A.T. was partially supported by NSERC}

\begin{document}

\begin{abstract}
Let $A$ be a unital simple separable C$^*$-algebra with strict comparison of positive elements.  We prove that the Cuntz semigroup of $A$ is recovered functorially from the Murray-von Neumann semigroup and the tracial state space $\mathrm{T}(A)$ whenever the extreme boundary of $\mathrm{T}(A)$ is compact and of finite covering dimension.  Combined with a result of Winter, we obtain $\mathcal{Z} \otimes A \cong A$ whenever $A$ moreover has locally finite decomposition rank.  As a corollary, we confirm Elliott's classification conjecture under reasonably general hypotheses which, notably, do not require any inductive limit structure.  These results all stem from our investigation of a basic question:  what are the possible ranks of operators in a unital simple C$^*$-algebra?
\end{abstract}

\maketitle

 \section{Introduction and statement of main results}\label{intro}

 The notion of rank for operators in C$^*$-algebras is of fundamental importance.  Indeed, one may view Murray and von Neumann's type classification of factors as a complete answer to the question of which ranks occur for projections in a factor.  For separable simple C$^*$-algebras the question of which ranks may occur has received rather less attention.  This is unfortunate given the many interesting areas it bears upon:  calculating the Cuntz semigroup, the question of whether or not a C$^*$-algebra is $\mathcal{Z}$-stable, the classification of nuclear simple C$^*$-algebras by $\mathrm{K}$-theory, and the classification of countably generated Hilbert modules, to name a few.  In this article we explore the rank question for a general tracial simple C$^*$-algebra, and give applications in each of these areas.  
 

What do we mean by the rank of an operator $x$ in a unital C$^*$-algebra $A$?  It is reasonable to assume that whatever the definition, $x$ and $x^*x$ should have the same rank.  We therefore consider only positive elements of $A$.  Let  $\tau:A \to \mathbb{C}$ be a tracial state.  For a positive element $a$ of $A$ we define
\[
d_\tau(a) = \lim_{n \to \infty} \tau(a^{1/n}).
\]
This defines the rank of $a$ with respect to $\tau$.  Our notion of rank for $a$ is then the  {\it rank function} of $a$, a map 
$\iota(a)$ from the tracial state space $\mathrm{T}(A)$ to $\mathbb{R}^+$ given by the formula $\iota(a)(\tau) = d_\tau(a)$.  (The maps $d_\tau$ and $\iota$ extend naturally to positive elements in $A \otimes \mathcal{K}$, and we always take this set of positive elements to be the domain of $\iota$.)  These rank functions are lower semicontinuous, affine, and nonnegative;  if $A$ is simple, then they are strictly positive.  We will be working extensively with various subsets of affine functions on $\mathrm{T}(A)$, so let us fix some convenient notation:  $\mathrm{Aff}(K)$ denotes the set of real-valued continuous affine functions on a (typically compact metrizable) Choquet simplex $K$, $\mathrm{LAff}(K)$ denotes the set of  bounded, strictly positive, lower semicontinuous and affine functions on $K$, and $\mathrm{SAff}(K)$ denotes the set of extended real-valued functions which can be obtained as the pointwise supremum of an increasing sequence from $\mathrm{LAff}(K)$.  The range question for ranks of operators in a unital simple tracial C$^*$-algebra $A$ can then be phrased as follows:

\vspace{2mm}
\begin{center}
When is the range of $\iota$ equal to all of $\mathrm{SAff}(\mathrm{T}(A))$?
\end{center}
\vspace{2mm}


\noindent
This brings us to our main result.

\begin{thms}\label{main}
Let $A$ be a unital  simple separable infinite-dimensional C$^*$-algebra with nonempty tracial state space $\mathrm{T}(A)$. 
Consider the following three conditions:
\begin{enumerate}
\item[(i)] The extreme boundary of $\mathrm{T}(A)$ is compact and of finite covering dimension.
\item[(ii)] The extreme boundary of $\mathrm{T}(A)$ is compact, and for each $n \in \mathbb{N}$ there is a positive $a \in A \otimes \mathcal{K}$ with the property that
\[
n d_\tau(a_n) \leq 1 \leq (n+1) d_\tau(a), \ \forall \tau \in \mathrm{T}(A).
\]
\item[(iii)] For each positive $b \in A \otimes \mathcal{K}$ such that $d_\tau(b) < \infty, \ \forall \tau \in \mathrm{T}(A)$, and each $\delta >0$, there is a positive $c \in A \otimes \mathcal{K}$ such that 
\[
|2 d_\tau(c) - d_\tau(b)| < \delta, \ \forall \tau \in \mathrm{T}(A).
\]
\end{enumerate}
If (ii) or (iii) is satisfied, then the uniform closure of the range of the map $\iota$ contains $\mathrm{Aff}(\mathrm{T}(A))$.
If any of (i)-(iii) is satisfied and $A$ moreover has strict comparison of positive elements, then the range of $\iota$ is all of $\mathrm{SAff}(\mathrm{T}(A))$.
\end{thms}
 \noindent
 Some remarks are in order.
 \begin{itemize}
 \item[$\bullet$] Finite-dimensional approximation properties such as nuclearity or exactness are not required by Theorem \ref{main}.
 \item[$\bullet$] In the presence of strict comparison, conditions (i), (ii), and (iii) are equivalent to the range of $\iota$ being all of $\mathrm{SAff}(\mathrm{T}(A))$. 
 \item[$\bullet$] Condition (ii) asks for the existence of positive operators with "almost constant rank".  These, we shall prove, are connected to the existence of unital $*$-homomorphisms from so-called dimension drop algebras in to $A$, and so to the important property of $\mathcal{Z}$-stability.  These operators exist, for example, in any
 crossed product of the form $\mathrm{C}(X) \rtimes G$, where $X$ is an infinite, finite dimensional compact space and at least one element of $G$ acts by a minimal homeomorphism (\cite{TW4}), and in many other C$^*$-algebras whose finer structure is presently out of reach. 
 \item[$\bullet$] 
 Condition (iii) is automatically satisfied in $\mathcal{Z}$-stable C$^*$-algebras (\cite{PT}) and in unital simple ASH algebras with slow dimension growth \cite{To101}.
 \item[$\bullet$] It is possible that $\iota( (A \otimes \mathcal{K})_+) = \mathrm{SAff}(\mathrm{T}(A))$ for any unital simple separable infinite-dimensional C$^*$-algebra with nonempty tracial state space.  Conditions (ii) and (iii) guarantee that the range of $\iota$ is dense in the sense that any element of $\mathrm{SAff}(\mathrm{T}(A))$ is a supremum of an increasing sequence from $\iota( (A \otimes \mathcal{K})_+)$.  In the absence of strict comparison, however, the elements of $(A \otimes \mathcal{K})_+$ giving rise to this sequence are not themselves increasing for the Cuntz relation, and so one cannot readily find a positive $a \in A \otimes \mathcal{K}$ 
for which $\iota(a)$ is the supremum of the said sequence.   
 \end{itemize}

 Theorem \ref{main} has several consequences for the algebras it covers.  We describe them briefly here, and in detail in Section \ref{apps}.  If $A$ as in Theorem \ref{main} satisfies any of (i)-(iii) and has strict comparison, then the Cuntz semigroup $\mathrm{Cu}(A)$ is recovered functorially from the Murray-von Neumann semigroup of $A$ and $\mathrm{T}(A)$, leading to the confirmation of two conjectures of Blackadar-Handelman concerning dimension functions (\cite{BPT}, \cite{bt}).  In fact, $\mathrm{Cu}(A) \cong \mathrm{Cu}(A \otimes \mathcal{Z})$, and so if $A$ moreover has locally finite decomposition rank---a mild condition satisfied by unital separable ASH algebras, for instance---then $A \cong A \otimes \mathcal{Z}$ by a result of Winter (\cite{NW}, \cite{Wi4}).  This leads to the complete classification of countably generated Hilbert modules over $A$ up to isomorphism via $\mathrm{K}_0(A)$ and $\mathrm{T}(A)$ in a manner analogous to the classification of W$^*$-modules over a $\mathrm{II}_1$ factor (\cite{bt}).  Finally, if $\mathcal{C}$ is the class of all such $A$ which, additionally, have projections separating traces, then $\mathcal{C}$ satisfies Elliott's classification conjecture:  the members of $\mathcal{C}$ are determined up to $*$-isomorphism by their graded ordered $\mathrm{K}$-theory (\cite{Wi2}, \cite{Wi3}, \cite{lin-niu}).

The paper is organized as follows:  Section \ref{prelim} reviews the Cuntz semigroup and dimension functions;  Section \ref{ipq} develops a criterion for embedding dimension drop algebras in C$^*$-algebras with strict comparison;  Section \ref{rankfunc} develops several techniques for constructing positive elements with prescribed rank function;  Section \ref{mainproof} contains the proof of Theorem \ref{main};  Section \ref{apps} details our applications.

 \section{Preliminaries}\label{prelim}

 Let $A$ be a C$^*$-algebra and let $\mathcal{K}$ denote the algebra of compact operators on a separable infinite-dimensional Hilbert space.  Let $(A \otimes \mathcal{K})_+$ denote the set of positive elements in $A \otimes \mathcal{K}$.  Given $a,b \in (A \otimes \mathcal{K})_+$,
we write $a \precsim b$ if there is a sequence $(v_n)$ in $A \otimes \mathcal{K}$ such that
\[
\| v_n b v_n^*-a\| \to 0.
\]
We then say
that $a$ is {\it Cuntz subequivalent} to $b$ (this relation restricts to usual Murray-von Neumann comparison on projections).  We write $a \sim b$ if $a \precsim b$ and $b \precsim a$, and say that $a$ is {\it Cuntz equivalent} to $b$.  Set $\mathrm{Cu}(A) = (A \otimes \mathcal{K})_+/\sim$, and write $\langle a \rangle$ for the equivalence class of $a$.
We equip $\mathrm{Cu}(A)$ with the binary operation
\[
\langle a \rangle + \langle b \rangle = \langle a \oplus b \rangle
\]
(using an isomorphism between $\mathrm{M}_2(\mathcal{K})$ and $\mathcal{K}$) and the partial order
\[
\langle a \rangle \leq \langle b \rangle \Leftrightarrow a \precsim b.
\]
This ordered Abelian semigroup is the {\it Cuntz semigroup} of $A$.

 It was shown in \cite{cei} that increasing sequences in $\mathrm{Cu}(A)$ always have a supremum, and we shall use this fact freely in the sequel.  Suppose now that $A$ is unital and $\tau:A \to \mathbb{C}$ is a tracial state.  The function $d_\tau$ introduced in Section \ref{intro} is constant on Cuntz equivalence classes, and drops to an additive order-preserving map on those classes in $\mathrm{Cu}(A)$ coming from positive elements in matrices over $A$.  This map has a unique supremum- and order-preserving extension to all of $\mathrm{Cu}(A)$, and we also denote this extension by $d_\tau$ (see \cite[Lemma 2.3]{bt}).
\begin{dfs}\label{strictcomp}
Let $A$ be a unital C$^*$-algebra.  We say that $A$ has strict comparison of positive elements (or simply strict comparison) if $a \precsim b$ for $a,b \in (A \otimes \mathcal{K})_+$ whenever
\[
d_\tau(a) < d_\tau(b), \, \forall\,\tau \in \{ \gamma \in \mathrm{T}(A) \ | \ d_\gamma(b) < \infty \}.
\]
\end{dfs}

 For positive elements $a,b$ in a C*-algebra $A$ we write that $a\approx b$ if here is $x\in A$ such that $x^*x=a$ and $xx^*=b$. The relation $\approx$ is
an equivalence relation \cite{Ped:factor}, sometimes referred to as {\it Cuntz-Pedersen equivalence}.
It is known that $a \approx b$ implies $a \sim b$.

If $a\in A $ is a positive element and $\tau\in \mathrm{T}(A)$ we denote by
$\mu_\tau$  the measure induced on the spectrum $\sigma(a)$ of $a$ by $\tau$.
Then $d_\tau(a)=\mu_\tau((0,\infty)\cap\sigma(a))$ and more generally
$$d_\tau(f(a))=\mu_\tau(\{t\in\sigma(a):\, f(t)>0\})$$
for all nonnegative continuous functions $f$ defined of $\sigma(a)$.

 \section{A criterion for embedding quotients of dimension drop algebras}\label{ipq}

The Jiang-Su algebra $\mathcal{Z}$ is an important object in the structure theory of separable nuclear C$^*$-algebras.  One wants to know when a given algebra $A$ has the property that $A \otimes \mathcal{Z} \cong A$, as one can then frequently obtain detailed information about $A$ through its $\mathrm{K}$-theory and positive tracial functionals.  If $A$ is separable, then $A \otimes \mathcal{Z} \cong A$ whenever one can find, for any natural number $m>1$, an approximately central sequence $(\phi_n)$ of unital $*$-homomorphisms from the prime dimension drop algebra
\[
\mathrm{I}_{m,m+1} := \{ f \in \mathrm{C}([0,1]; \mathrm{M}_m \otimes \mathrm{M}_{m+1}) \ | \ 
f(0) \in \mathbf{1}_m \otimes \mathrm{M}_{m+1}, \ f(1) \in \mathrm{M}_m \otimes \mathbf{1}_{m+1}\}
\]
into $A$.  It is therefore of interest to characterize when $A$ admits a unital $*$-homomorphism $\phi:\mathrm{I}_{m,m+1} \to A$.  In this section we obtain such a characterization in the case that $A$ has strict comparison (Theorem \ref{embedcriterion}).  Our result should be compared with \cite[Proposition 5.1]{RorWin:JiangSu}, which uses the assumption of stable rank one instead of strict comparison.  Some of the results we develop here will be employed later to construct positive elements with specified rank functions.
 
\begin{lms}\label{sup1}
Let $A$ be a unital C$^*$-algebra with $\mathrm{T}(A)\neq \emptyset$, and let $a \in M_k(A)$ be positive.  Suppose that there are $0<\alpha <\beta $ such that
$\alpha<d_\tau(a) < \beta $ for every $\tau$ in a closed subset $X$ of $\mathrm{T}(A)$.  Then there exist $\epsilon >0$ and an open neigbourhood $U$ of $X$ with the property that for
$
\alpha<d_\tau((a-\epsilon)_+) <\beta , \ \forall \tau \in U.
$
\end{lms}

\begin{proof}
  Since $d_\tau((a-\epsilon)_+) \nearrow d_\tau(a)$ as
$\epsilon \searrow 0$ for each $\tau$, we can fix $\epsilon_\tau > 0$ such that $d_\tau((a-\epsilon_\tau)_+) > \alpha$;  since
$\gamma \mapsto d_\gamma((a-\epsilon_\tau)_+)$ is lower semicontinuous, we can find an open neighbourhood $V_\tau$ of
$\tau$ such that
\[
d_\gamma((a-\epsilon_\tau)_+) > \alpha, \ \forall \gamma \in V_\tau.
\]
The family $\{V_\tau \}_{\tau \in X}$ is an open cover of $X$, and so
\mbox{$X \subset V_{\tau_1} \cup \cdots \cup V_{\tau_n}$} for some
$\tau_1,\ldots,\tau_n \in X$.
  Set $\epsilon := \min \{\epsilon_{\tau_1},\ldots,\epsilon_{\tau_n}\}$ and
  $V:=V_{\tau_1}\cup \cdots \cup V_{\tau_n}$, so that
for each $\tau \in V$, $\tau \in V_{\tau_i}$ for some $i$, and we have
\[
d_\tau((a-\epsilon)_+) \geq d_{\tau}((a-\epsilon_{\tau_i})_+) > \alpha.
\]

Let $\mu_\tau$ be the measure induced on $\sigma(a)$ by $\tau$.
We also have
\[
d_\tau((a-\epsilon)_+) = \mu_\tau((\epsilon,\infty) \cap \sigma(a)) \leq \mu_\tau([\epsilon,\infty) \cap \sigma(a)) \leq
\mu_\tau((0,\infty) \cap \sigma(a)) \leq d_\tau(a), \ \forall \tau \in \mathrm{T}(A).
\]
In particular, $d_\tau((a-\epsilon)_+) \leq \mu_\tau([\epsilon,\infty) \cap \sigma(a)) < \beta$ for all $\tau \in X$.  By the Portmanteau theorem, the map $\gamma \mapsto \mu_\tau([\epsilon,\infty)
\cap \sigma(a))$ is upper semicontinuous, and so  the set
\[W=\{\gamma\in \mathrm{T}(A):  \mu_\gamma([\epsilon,\infty) \cap \sigma(a)) < \beta\}\]
is open and contains $X$.  Moreover, for any $\gamma \in W$, 
  we have
\[
d_\gamma((a-\epsilon)_+)  <\beta.
\]

We conclude the proof by setting  $U = V \cap W$.
\end{proof}
For each $\eta>0$ define a continuous map $f_\eta:\mathbb{R}_{+} \to [0,1]$ by the following formula:
\begin{equation}
 \label{feta}
f_\eta(t) = \left\{ \begin{array}{ll}  t/\eta, & 0 < t < \eta \\ 1, & t \geq \eta. \end{array} \right.
\end{equation}
 \begin{lms}\label{almostconstant}
 Let $A$ be a unital separable C$^*$-algebra with nonempty tracial state space, and let $X \subseteq \mathrm{T}(A)$ be closed.   Suppose
 that $a \in \mathrm{M}_k(A)$ is a positive element with the property that
 \[
 \beta - \alpha < d_\tau(a) \leq \beta, \ \forall \tau \in X
 \]
 for some $0 < \alpha < \beta$.
 Then there is $\eta > 0$ such that
\[ k-\beta \leq d_\tau(1-f_\eta(a))<k-\beta+2\alpha, \ \forall \tau \in X.
\]
 \end{lms}

 \begin{proof} Let $\mu_\tau$ be the measure induced on $\sigma(a)$ by $\tau$.
 By Lemma ~\ref{sup1} there is $\eta>0$ such that $d_\tau ((a-\eta)_+)>\beta-\alpha$ for all $\tau \in X$. Therefore
 \[
d_\tau((a-\eta)_+)= \mu_\tau\left( (\eta,\infty)\cap \sigma(a) \right) >\beta-\alpha, \ \forall \tau \in X.
 \]
 It follows that
 \[\mu_\tau((0,\eta])=d_\tau(a)-\mu_\tau((\eta,\infty)\cap \sigma(a))<\beta-(\beta-\alpha)=
 \alpha,   \ \forall \tau \in X.\]
 Then, $d_\tau(1-f_\eta(a))=\mu_\tau ([0,\eta)\cap \sigma(a))=\mu_\tau ((0,\eta)\cap \sigma(a))+\mu_\tau (\{0\}\cap \sigma(a))$ and hence
\[d_\tau(1-f_\eta(a))=\mu_\tau ((0,\eta)\cap \sigma(a))+k-d_\tau(a)<\alpha+k-(\beta-\alpha)
=k-\beta+2\alpha, \ \ \forall \tau \in X.\]
Moreover, $d_\tau(1-f_\eta(a))\geq k-d_\tau(a) \geq k-\beta$.
 \end{proof}

 \begin{lms}
  \label{PedCuntz_CC}
  Let $a,b$ be positive elements in a C*-algebra $A$. If $a\approx b$ then $f(a)\approx f(b)$
  for any continuous function $f:[0,\infty)\to [0,\infty)$ with $f(0)=0$.
 \end{lms}
\begin{proof}
 By assumption there is $x\in A$ such that $x^*x=a$ and $xx^*=b$. Let $x=v|x|$ be the polar decomposition
 of $x$ where $v$ is a partial isometry in the enveloping von Neumann algebra $A^{**}$.
 Then as in \cite[Lemma 2.4]{KR1} the map $d\mapsto vdv^*$ defines a isomorphism from $\overline{aAa}$ to
 $\overline{bAb}$ which maps $a$ to $b$ and hence $f(a)$ to $f(b)$. Therefore $vf(a)v^*=f(b)$.
 Let us note that $y:=vf(a)^{1/2}$ is an element of $A$.
 Indeed, since $f(0)=0$, $vf(a)^{1/2}\in \overline{va^{1/2}A}=\overline{v|x|A}=\overline{xA}$.
 It follows that
\[
 f(b) = v f(a) v^* =yy^*\approx y^*y=f(a)^{1/2} v^*v f(a)^{1/2} = f(a).
\]
as required.
\end{proof}

 \begin{lms}\label{orthoelements}
 Let $A$ be a unital separable C$^*$-algebra with strict comparison of positive elements.  Also suppose that for each $m \in \mathbb{N}$, there is
 $x \in \mathrm{Cu}(A)$ such that
 \[
 m d_\tau(x) \leq 1 \leq (m+1) d_\tau(x), \ \forall \tau \in \mathrm{T}(A)\neq \emptyset.
 \]
 It follows that for each
 $n \in \mathbb{N}$ and for any $0<\epsilon <1/n$, there exist mutually orthogonal positive elements $a_1,\ldots,a_n \in A$ with the following properties:
 \begin{enumerate}
 \item[(i)] $a_i \approx a_j$ for all $i,j \in \{1,\ldots,n\}$;
 \item[(ii)] $1/n - \epsilon < d_\tau(a_i) < 1/n$ for each $\tau \in \mathrm{T}(A)$ and $i \in \{1,\ldots,n\}$.
 \end{enumerate}
 \end{lms}

 \begin{proof}
 Let us first show that for any $\epsilon>0$, there is  $a$ in $M_k(A)$, $k\geq 1$,  with the property that
 \begin{equation}\label{aineq}
 1/n -\epsilon < d_\tau(a) < 1/n, \ \forall \tau \in \mathrm{T}(A).
 \end{equation}
 Let $r,m$ be a natural numbers such that $[r/(m+1),r/m]\subset (1/n -\epsilon ,1/n)$.
 By hypothesis there is $x\in \mathrm{Cu}(A)$  such that $d_\tau(x) \in [1/(m+1),1/m]$ for all $\tau\in \mathrm{T}(A)$. It follows that  $d_\tau(rx) \in [r/(m+1),r/m]$ for all $\tau\in \mathrm{T}(A)$.
 Let $a\in (A \otimes \mathcal{K})_+$ representing $rx \in \mathrm{Cu}(A)$. Then $a$ satisfies \eqref{aineq}. By \cite[Lemma~2.2]{bt} we may arrange that $a\in M_k(A)$ for some  $k\geq 1$.

 We will now prove the Lemma by induction.  Let $n$ and $\epsilon>0$ be given.    Find $a\in M_k(A)$ satisfying (\ref{aineq}).    By Lemma \ref{sup1} there is  $\eta > 0$
 such that
 \[
 1/n - \epsilon < d_\tau( (a-\eta)_+) < 1/n, \ \forall \tau \in \mathrm{T}(A).
 \]
Since $d_\tau(a) < 1/n \leq 1$ for each $\tau \in \mathrm{T}(A)$, we have
 $a \precsim 1_A$ by strict comparison. There is $v$ in  $M_k(A)$ such that $v1_Av^* = (a-\eta)_+$, whence $A \ni a_1 := 1_A v^*v 1_A \approx (a-\eta)_+$
 also satisfies
 \[
  1/n - \epsilon  < d_\tau(a_1) < 1/n, \ \forall \tau \in \mathrm{T}(A).
 \]
 This proves the Lemma in the case $n=1$; for larger $n$ we proceed inductively.

Suppose that for some $k<n$ and $0<\epsilon <1/n$ we have found mutually orthogonal positive elements $a_1,\ldots,a_k \in A$ with the following properties:
\begin{enumerate}
\item[(i)] $a_i \approx a_j$ for all $i,j \in \{1,\ldots,k\}$;
\item[(ii)] $1/n - \epsilon  < d_\tau(a_i) < 1/n$ for each $i \in \{1,\ldots,k\}$ and $\tau \in \mathrm{T}(A)$.
\end{enumerate}
We will explain how to use these $a_i$ to construct mutually orthogonal positive elements $\tilde{a}_1,\ldots,\tilde{a}_{k+1}$
in $A$ which
satisfy (i) and (ii) above with $k$ replaced by $k+1$.  Repeated application of this construction yields the Lemma in full.

By Lemma \ref{sup1} there is $\eta> 0$ such that
\[
1/n - \epsilon < d_\tau((a_i - \eta)_+) < 1/n, \ \forall \tau \in \mathrm{T}(A), \ \forall i \in \{1,\ldots,k\}.
\]
Set $a:=\sum_{i=1}^k a_i$ and  $c:= 1-f_\eta(a)$, where $f_\eta$ is given by \eqref{feta}.
Note that $c$ is orthogonal to $(a-\eta)_+$ and hence to each $(a_i-\eta)_+$.
 For each $\tau \in \mathrm{T}(A)$,
\[
  d_\tau(a)
 =  \sum_{i=1}^k d_\tau(a_i)\in (k/n-k\epsilon,k/n).
\]

By Lemma~\ref{almostconstant} it follows that for all $\tau\in \mathrm{T}(A)$
 \[
 d_\tau(c) = d_\tau(1-f_\eta(a))\geq 1-k/n \geq 1/n.
 \]

We are in a position to construct $\tilde{a}_1,\ldots,\tilde{a}_{k+1} \in A$.  Since $d_\tau(a_1)<1/n
< d_\tau(c)$ for all $\tau\in \mathrm{T}(A)$, it follows that $a_1\precsim c$ by strict comparison.
Therefore there exists $w \in A$ such that
$(a_1-\eta)_+=wcw^*$.  Set  $\tilde{a}_i:=
(a_i-\eta)_{+}$ for $ i \in \{1,\ldots,k\}$ and $\tilde{a}_{k+1} := c^{1/2}w^*wc^{1/2} \approx (a_1-\eta)_+$.
By Lemma~\ref{PedCuntz_CC} we have
\[
\tilde{a}_i =(a_i - \eta)_+\approx (a_1 - \eta)_+, \ \forall i \in \{1,\ldots,k\}.
\]
This establishes (i).  Our choice of $\eta$ establishes (ii).  The mutual orthogonality of the $\tilde{a}_i$ follows
from the orthogonality of $c$ to all of $(a_i - \eta)_+$, $i \in \{1,\ldots,k\}$, and
the fact that $\tilde{a}_{k+1} \in \overline{cAc}$.
 \end{proof}

 \begin{lms}\label{cuntztomvn2}
Let $A$ be a separable unital C$^*$-algebra with strict comparison of positive elements.  Suppose that $b_1,\ldots,b_n$, $n\geq 2$, are orthogonal positive elements in $A$ with the following properties:
\begin{enumerate}
\item[(i)] $b_i \approx b_j$ for each $i,j \in \{1,\ldots,n\}$;
\item[(ii)] $ 1/n - 1/3n^2 < d_\tau(b_i) < 1/n$ for all $\tau \in \mathrm{T}(A)\neq \emptyset$ and each $i \in \{1,\ldots,n\}$.
\end{enumerate}
It follows that there is a unital $*$-homomorphism $\gamma:\mathrm{I}_{n,n+1} \to A$.
\end{lms}

\begin{proof}
By (\cite[Prop.~5.1]{RorWin:JiangSu}), it will suffice to find orthogonal positive elements
$a_1,\ldots,a_n$ of $A$ such that $a_i \approx a_j$ for each $i,j \in \{1,\ldots,n\}$ and
$(1- \sum_{i=1}^n a_i) \precsim (a_1-\epsilon)_+$ for some $\epsilon>0$.
The latter condition can be replaced by
$d_\tau(1- \sum_{i=1}^n a_i) < d_\tau((a_1-\epsilon)_+)$, $\forall \tau \in \mathrm{T}(A)$,
due to the strict comparison assumption.

Let $b_1,\ldots,b_n$ be as specified in the statement of the Lemma and
set $b=\sum_{i=1}^n b_i $.
In view of the mutual orthogonality of the $b_i$,  (ii) implies
\begin{equation}
 1-1/3n<d_\tau(b)<1,\quad \forall \tau \in \mathrm{T}(A).
\end{equation}
Apply Lemma \ref{almostconstant}  to find $\eta>0$ such that
\[d_\tau(1-f_\eta(b))<2/3n\]
where $f_\eta$ is defined by \eqref{feta}.
Set $a_i = f_\eta(b_i)$ and note that
 the mutual orthogonality of the $b_i$ implies that
\[
f_\eta( b) = \sum_{i=1}^n f_\eta(b_i) =\sum_{i=1}^n a_i.
\]
 Since $d_\tau(a_1)=d_\tau(b_1)$, by Lemma~\ref{sup1} there is $\epsilon>0$ such that
  $1/n-1/3n^2 < d_\tau((a_1-\epsilon)_+)$ for all $\tau \in \mathrm{T}(A)$.
Now for each $\tau \in \mathrm{T}(A)$ we have
\[
d_\tau\left(1- \sum_i^n a_i\right)
 =  d_\tau\left(1- f_\eta(b)\right)
 <  2/3n
 \leq  1/n-1/3n^2< d_\tau((a_1-\epsilon)_+).
\]

To complete the proof of the Lemma, we observe that $a_i \approx a_j$ for each $i,j \in \{1,\ldots,n\}$
as a consequence of Lemma~\ref{PedCuntz_CC}.
 \end{proof}

\begin{thms}\label{embedcriterion}
 Let $A$ be a separable unital C$^*$-algebra with strict comparison of positive elements and nonempty tracial state space.  The following statements are equivalent:
 \begin{enumerate}
 \item[(i)] for each $m \in \mathbb{N}$, there is $x \in \mathrm{Cu}(A)$ such that
 \[
 mx \leq \langle 1_A \rangle \leq (m+1) x;
 \]
 \item[(ii)] for each $m \in \mathbb{N}$, there is $x \in \mathrm{Cu}(A)$ such that
 \[
 m d_\tau(x) \leq 1 \leq (m+1) d_\tau(x), \ \forall \tau \in \mathrm{T}(A);
 \]
 \item[(iii)] for each $m \in \mathbb{N}$, there is a unital $*$-homomorphism $\phi_m:\mathrm{I}_{m,m+1} \to A$.
 \end{enumerate}
 \end{thms}

 \begin{proof}
 For $\mathrm{(i)} \Rightarrow \mathrm{(ii)}$, we apply the order preserving state $d_\tau$ to the inequality
 \[
 mx \leq \langle 1_A \rangle \leq (m+1) x_.
 \]
 The implication $\mathrm{(ii)} \Rightarrow \mathrm{(iii)}$ is the combination of Lemmas \ref{orthoelements} and \ref{cuntztomvn2}.
 Finally, $\mathrm{(iii)} \Rightarrow \mathrm{(i)}$ is due to R\o rdam \cite{R4}.
 \end{proof}

\section{Rank functions on trace spaces}\label{rankfunc}
\begin{lms}\label{almostdelta2}
Let $A$ be a unital simple separable infinite-dimensional C$^*$-algebra and $\tau$ a normalised trace on $A$.
Let $0<s<r$  be given.  It follows that there are an open neighbourhood $U$ of $\tau$ in $\mathrm{T}(A)$
and a positive element $a$ in some $\mathrm{M}_k(a)$ such that
\[
s < d_\gamma(a) < r, \ \forall \gamma \in U.
\]
\end{lms}

\begin{proof}
Let $\tau$ and $0<s<r$ as in the hypotheses of the Lemma be given. 
 Since $A$ is infinite-dimensional, there is $b \in A_+$ with zero as an accumulation point of its spectrum.  For each $n\in \mathbb{N}$,  let $f_n$ be a positive continuous function with support $(0,1/n)$ and set $b_n:=f_n(b)$.  It follows that
 \[
 0 < d_\tau(b_n) = \mu_\tau( (0,1/n) \cap \sigma(b)) \stackrel{n \to \infty}{\longrightarrow} 0.
 \]
 Choose $n \in \mathbb{N}$ large enough such that $d_\tau(b_n)<r-s$. Then
 \[
 s < k d_\tau(b_n) < r
 \]
 for some $k \in \mathbb{N}$.  If we set $b:= \oplus_{i=1}^k b_n \in \mathrm{M}_k(A)$, then
 $d_\tau(b) = d_\tau \left( \oplus_{i=1}^k b_n \right) = k d_\tau(b_n)$ and hence 
 \[s < d_\tau(b) < r.\]
We conclude the proof by applying Lemma~\ref{sup1} with $X=\{\tau\}$.
 \end{proof}

 \begin{lms}\label{compression}
 Let $A$ be a unital  C$^*$-algebra and $\tau$ a normalised trace on $A$.  Let $x,y$ be positive elements  in  $\mathrm{M}_k(A)$. Then $d_\tau(y^{1/2}x\,y^{1/2}) \geq d_\tau(x)-d_\tau(1_k-y)$, where $1_k$ denotes the unit of $\mathrm{M}_k(A)$.
 \end{lms}
 \begin{proof}
   We have
  \begin{eqnarray*}
x & = & x^{1/2} y \,x^{1/2} + x^{1/2}(1_k-y)x^{1/2} \\
& \precsim & x^{1/2} y\, x^{1/2} \oplus x^{1/2}(1_k-y)x^{1/2} \\
& \approx & y^{1/2}x \,y^{1/2} \oplus x^{1/2}(1_k-y)x^{1/2} \\
& \precsim & y^{1/2}x \,y^{1/2} \oplus (1_k-y).
\end{eqnarray*}
Applying $d_\tau$ we obtain $ d_\tau(x) \leq d_\tau(y^{1/2}x\,y^{1/2})+d_\tau(1_k-y)$.
 \end{proof}

 We recall a result from \cite{lintai} based on work of Cuntz and Pedersen:

 \begin{thms}\label{hitf}(\cite[Thm.~9.3]{lintai})
 Let $A$ be a unital simple C$^*$-algebra with nonempty tracial state space, and let $f$ be a strictly positive affine continuous function on $\mathrm{T}(A)$.  It follows that for any $\epsilon>0$ there is a positive element $a$ of $A$ such that $f(\tau) = \tau(a), \ \forall \tau \in  \mathrm{T}(A)$, and $\|a\| < \|f\| + \epsilon$.
 \end{thms}

\begin{lms}\label{almostdelta1}
Let $A$ be a separable unital simple infinite-dimensional C$^*$-algebra whose tracial state space has compact extreme boundary $X=\partial_e\mathrm{T}(A)$.
It follows that for any $\delta > 0$ and closed set $Y \subseteq X$  there is a nonzero positive element $a$ of $A$ with the property that
$d_{\tau}(a) < \delta, \ \forall \tau \in Y$
and $d_{\tau}(a) = 1, \ \forall \tau \in X \backslash Y$.\end{lms}

\begin{proof}
Let $\delta$ and $Y$ be given.
Let us assume first that $Y \neq X$.
Fix a decreasing sequence $\{U_n\}_{n=2}^\infty$ of open subsets of $X$ with the property that
$Y = \cap_{n=2}^\infty  U_n$ and $U_n^c \neq \emptyset$.  The compactness of $X$ implies that
the natural restriction map $\mathrm{Aff}(\mathrm{T}(A))\to C(\partial_e\mathrm{T}(A),\mathbb{R})$ is an isometric isomorphism of Banach spaces,  \cite[Cor.~11.20]{Go:book}.  We may therefore use Theorem \ref{hitf} to
produce,  a sequence $(b_n)_{n=2}^\infty$ in $A_+$ with the following properties:
\begin{enumerate}
\item[(i)] $\tau(b_n) > 1-1/n, \ \forall \tau \in U_n^c$;
\item[(ii)] $\tau(b_n) < \delta/2^{n}n, \ \forall \tau \in Y$;
\item[(iii)] $\|b_n\| \leq 1$.
\end{enumerate}

 For any $\tau \in U_n^c$ we have
\[
d_\tau((b_n-1/n)_+)  \stackrel{(iii)}\geq  \tau((b_n-1/n)_+)
 \geq  \tau(b_n)-1/n
\stackrel{(i)}{>}  1-2/n.\]
In particular $(b_n-1/n)_+\neq 0$. Moreover, for any $\tau \in Y$ we have
\[ d_\tau((b_n-1/n)_+) =
n\int \frac{1}{n} \chi_{(1/n,\infty)} \ d\mu_\tau\leq n\tau(b_n)\stackrel{(ii)}{<}\delta/ 2^{n}.\]

Set $c_n:=2^{-n}(b_n-1/n)_+$, so that $d_\tau(c_n) > 1- 2/n$ for each $\tau \in U_n^c$,  $d_\tau(c_n) <  \delta/2^n$ for each
$\tau \in Y$ and $\|c_n\|\leq 2^{-n}$.  Now set
\[
a:= \sum_{n=2}^\infty c_n \in A_+.
\]
If $\tau \in Y$, then, using the lower semicontinuity of $d_\tau$, we have
\begin{eqnarray*}
d_\tau(a) & \leq & \liminf_k d_\tau\left( \sum_{n=2}^k c_n \right) \\
& \leq & \liminf_k \sum_{n=2}^k d_\tau(c_n) \\
& < & \delta.
\end{eqnarray*}

If $\tau \in X \backslash Y$, then $\tau \in U_k^c$ for all $k$ sufficiently large.  It follows that for these same $k$,
\[
d_\tau(a) = d_\tau \left( \sum_{n=2}^\infty c_n \right) \geq   d_\tau(c_k) \geq 1-2/k.
\]
We conclude that $d_\tau(a) \geq 1$ for each such $\tau$.  On the other hand, $a \in A$, so $d_\tau(a) \leq 1$ for any
$\tau \in \mathrm{T}(A)$.  This completes the proof in the case that $Y \neq X$.

If $Y = X$ and $X$ is a singleton, then the Theorem follows from Lemma \ref{almostdelta2};  we assume that $X$ contains
at least two points, say $\tau, \gamma$.  Since $X$ is Hausdorff we can find open subsets $U$ and $V$ of $X$ such that $\tau \in U$,
$\gamma \in V$, and $U \cap V = \emptyset$.  Use the case of the Theorem established above twice, with $Y$ replaced by $U^c$ and $V^c$,
to obtain positive elements $b,c$ of $A$ with the following properties:
\[
d_\nu(c) = 1, \ \forall \nu \in U, \ \ \mathrm{and} \ \ d_\nu(c)< \delta, \ \forall \nu \in X \backslash U;
\]
\[
d_\nu(b) = 1, \ \forall \nu \in V, \ \ \mathrm{and} \ \ d_\nu(b)< \delta, \ \forall \nu \in X \backslash V.
\]
 It follows by Lemma \ref{almostconstant} applied to the element $b\in A$ and the compact set $\{\gamma\}$,
with $k=1$, $\beta=1$ and $\alpha=d_\gamma(c)/2>0$,
that there is $\eta>0$ such that $d_\gamma(1-f_\eta(b))  < 2\alpha=d_\gamma(c)$, where  $f_\eta$ is the
function given in \eqref{feta}.

Since $d_\nu(b) = d_\nu(f_\eta(b))$ for any $\nu \in \mathrm{T}(A)$, we may replace $b$ with $f_\eta(b)$ and hence
arrange that
\begin{equation}\label{bkappa}
d_\gamma(1-b)<d_\tau(c).
\end{equation}

Set $a = b^{1/2}cb^{1/2}$.  Since $a \precsim c$, we have $d_\nu(a) < \delta$ for every $\nu \in X \backslash U$.
Also, $a  \approx c^{1/2} b c^{1/2} \precsim b$, so $d_\nu(a)<\delta$ for every $\nu \in X \backslash V$.
Since $U$ and $V$ are disjoint we conclude that $d_\nu(a) < \delta$ for every $\nu \in X$.  It remains to prove that
$a \neq 0$. To this purpose we show that $d_\gamma(a) > 0$. By Lemma~\ref{compression} and  (\ref{bkappa}) above we have
\[
d_\gamma(a) = d_\gamma(b^{1/2}c \,b^{1/2})\geq d_\gamma(c) - d_\gamma(1-b) >0.
\]
\end{proof}

\begin{lms}\label{cutoff}
Let $A$ be a unital simple separable infinite-dimensional C$^*$-algebra.  Suppose that the extreme boundary $X$ of $\mathrm{T}(A)$
is compact and nonempty.  Let $a \in A$ be positive, and let there be given an open subset $U$ of $X$ and $\delta>0$.  It follows that there is a
positive element $b$ of $A \otimes \mathcal{K}$ with the following properties:
\[
d_\tau(b) = d_\tau(a), \ \forall \tau \in U \ \ \mathrm{and} \ \ d_\tau(b) \leq \delta, \ \forall \tau \in X \backslash U.
\]
\end{lms}

\begin{proof}
 Use Lemma \ref{almostdelta1} to find a positive element $h$ of $A$ with the property that
\begin{equation}\label{hone}
d_\tau(h) < \delta, \ \forall \tau \in X \backslash U \ \ \mathrm{and} \ \ d_\tau(h)=1, \ \forall \tau \in U.
\end{equation}
Let $V_1 \subseteq V_2 \subseteq V_3 \subseteq \cdots$ be a sequence of open subsets of $X$ such that $\overline{V}_i
\subseteq U$ for each $i$ and $\cup_{i=1}^\infty V_i = U$.
Trivially,
\begin{equation}\label{htwo}
1-1/2i < d_\tau(h) \leq 1, \ \forall \tau \in \overline{V_i},
\end{equation}
 and so Lemma \ref{almostconstant} applied for $k=\beta=1$ and $\alpha=1/2i$ yields $\eta_i>0$ such that
\begin{equation}\label{hthree}
d_\tau(1-f_{\eta_i}(h))  < 1/i, \ \forall \tau \in \overline{V}_i.
\end{equation}
To simplify notation in the remainder of the proof, let us simply re-label $f_{\eta_i}(h)$ as $h_i$. We may assume that the sequence $(\eta_i)$ is decreasing  so that
the sequence $(h_i)$ is increasing.
Since $d_\tau(h) = d_\tau(f_\eta(h))$ for any $\tau \in \mathrm{T}(A)$ and $\eta>0$, it follows from \eqref{hone} that
\begin{equation}\label{hone-i}
d_\tau(h_i) < \delta, \ \forall \tau \in X \backslash U \ \ \mathrm{and} \ \ d_\tau(h_i)=1, \ \forall \tau \in U.
\end{equation}

Set $a_i:= h_i^{1/2}ah_i^{1/2}$.  Since $a_i \precsim a$, we have
\[
d_\tau(a_i) \leq d_\tau(a), \ \forall \tau \in U.
\]
Also, since $h_i^{1/2}ah_i^{1/2} \approx a^{1/2} h_i a^{1/2} \precsim h_i$, we have
\[
d_{\tau}(a_i)  \leq d_{\tau}(h_i) < \delta, \forall \tau \in X \backslash U.
\]
For our lower bound, we observe that by Lemma~\ref{compression} and (\ref{hthree}) we have
 for any $\tau \in \overline{V}_i$:
\begin{equation*}
d_{\tau}(a_i) =d_{\tau}(h_i^{1/2}ah_i^{1/2})\geq d_{\tau}(a) - d_{\tau}(1-h_i) > d_\tau(a) - 1/i.
\end{equation*}
Therefore we have
\begin{equation}\label{final}
d_\tau(a_{i}) < \delta, \ \forall \tau \in X \backslash U \ \ \mathrm{and} \ \ d_\tau(a) -1/i < d_\tau(a_{i}) < d_\tau(a),
\ \forall \tau \in \overline{V_i}.
\end{equation}
Since $h_{i} \leq h_{i+1}$ and
\[
a_{i} = h_{i}^{1/2} d h_{i}^{1/2} \precsim h_{i+1}^{1/2} d h_{i+1}^{1/2} = a_{i+1}.
\]
The increasing sequence $(\langle a_{i} \rangle )_{i=1}^{\infty}$ has a supremum $y$, where $y = \langle b \rangle$ for some positive element $b$ of $A \otimes \mathcal{K}$ by \cite{cei}.  Since each $d_\tau$ is a supremum preserving state on
$\mathrm{Cu}(A)$, we conclude from (\ref{final}) that
\[
d_\tau(b) \leq \delta, \ \forall \tau \in X \backslash U \ \ \mathrm{and} \ \ d_\tau(b) = d_\tau(a), \ \forall \tau \in U,
\]
as desired.


\end{proof}


 \begin{lms}\label{almostmax}
Let $A$ be a unital simple separable C$^*$-algebra with strict comparison of positive elements and at least one bounded trace.  Suppose that $X = \partial_e \mathrm{T}(A)$ is compact, and let $a,b \in A$ be positive.  Suppose that there are $0<\alpha< \beta  <\gamma \leq 1$ and open sets $U,V \subseteq X$ with the property that
\[
\alpha < d_\tau(a) < \beta, \ \forall \tau \in U \ \ \mathrm{and} \ \ \beta < d_\tau(b) < \gamma, \ \forall \tau \in V.
\]
It follows that for any closed set $K \subset U \cup V$, there is a positive element $c$ of $M_2(A)$ with the property that
\[
\alpha < d_\tau(c) < \gamma, \ \forall \tau \in K.
\]
\end{lms}

 \begin{proof}

Since the topology on $X$ is metrizable and $K$ is compact there exist closed subsets
$E \subseteq U$ and $F \subseteq V$ such that
 $K \subseteq E^\circ \cup F^ \circ \subseteq E \cup F$  and $E$ and $F$ are the closures of their interiors.

By Lemma~\ref{cutoff} and strict comparison we may assume that
 $d_\tau(a)<\beta$ for all $\tau \in X\setminus U$.
Apply Lemma \ref{sup1} to $a$ to find an $\epsilon>0$ such that $d_\tau((a-\epsilon)_+) > \alpha$ for each $\tau \in E$.  Likewise find $\delta > 0$ such that $d_\tau((b-\delta)_+) > \beta$ for each $\tau \in F$.
For   $\tau \in X$ we have
 \[
  d_\tau((b-\delta)_+) \leq \mu_\tau([\delta,\infty) \cap \sigma(b)) \leq d_\tau(b) < \gamma.
 \]
 The map $\tau \mapsto \mu_\tau([\delta,\infty) \cap \sigma(b))$ is upper semicontinuous on the compact set $X$, and so attains a
 maximum $\gamma_0 < \gamma$.  Use Lemma \ref{almostdelta1} to find a positive element $z \in A$ with the property that
 \[
 d_\tau(z) < \gamma-\gamma_0, \ \forall \tau \in F \ \ \mathrm{and} \ \ d_\tau(z) = 1, \ \forall \tau \in X \backslash F.
 \]
 Set $y = (b-\delta)_+ \oplus z \in \mathrm{M}_2(A)$, so that
 \[
 \beta < d_\tau(y) <\gamma, \ \forall \tau \in F \ \ \mathrm{and} \ \ d_\tau(y) \geq  \beta, \forall \tau \in X \backslash F.
 \]
 It follows that $d_\tau(y) > d_\tau(a)$ for every $\tau \in X$, whence $a \precsim y$ by strict comparison.  Using this we
 can find   $v \in \mathrm{M}_2(A)$ such that $ (a-\epsilon)_+=vyv^* $.  Set $x = y^{1/2}
 v^*v\,y^{1/2}$, so that $x \approx (a-\epsilon)_+$ and $x \in \overline{yM_2(A)y}$.
Moreover, $d_\tau(x)<\beta$ for all $\tau \in X$.

Choose $0<\lambda <\min\{\beta-\alpha, \gamma-\beta \}$.
Use Lemma \ref{almostdelta1} to find a positive element $h$  in $\mathrm{M}_2(A)$ with the following property:
\[
d_\tau(h) = 2, \ \forall \tau \in F^\circ \ \ \mathrm{and} \ \ d_\tau(h) < \lambda, \ \forall \tau \in X \backslash F^\circ.
\]

 We can find closed sets $E_1 \subseteq E^\circ$ and $F_1 \subseteq F^\circ$ such that $K \subseteq E_1 \cup F_1$.
Replacing $h$  with $f_\eta(h)$,  for sufficiently small $\eta$, we may arrange by
Lemma \ref{almostconstant} applied with $k=\beta=2$ and $\alpha=\lambda/2$,  that
\[
d_\tau(1_{2} - h) < \lambda, \ \forall \tau \in F_1.
\]

We define
\[
c = x + y^{1/2} h y^{1/2}.
\]
Let us prove that this definition yields the inequality required by the Lemma.  Let $\tau \in K$ be given.  If
$\tau \in E_1$, then
\[
d_\tau(c) \geq d_\tau(x) =d_\tau((a-\epsilon)_+) > \alpha.
\]
If $\tau \in F_1$, then, by Lemma \ref{compression}, we have
\[
d_\tau(c) \geq d_\tau(y^{1/2} h y^{1/2}) = d_\tau(h^{1/2} y h^{1/2}) \geq d_\tau(y) - d_\tau(1_2 - h) > \beta -\lambda > \alpha.
\]
Thus, $d_\tau(c) > \alpha$ for each $\tau \in K$.  For the upper bound, observe first that $c \in \overline{yM_2(A)y}$, whence
$d_\tau(c) \leq d_\tau(y)$ for every $\tau \in X$.  In particular,
\[
d_\tau(c) < \gamma, \ \forall \tau \in F.
\]
If $\tau \in K \backslash F$ then $\tau \in  X \backslash F^\circ$.  It follows that
\[
d_\tau(c) \leq d_\tau(x) + d_\tau(y^{1/2} h y^{1/2}) \leq d_\tau(x) + d_\tau(h)
< \beta + \lambda < \gamma.
\]
 \end{proof}

\section{The proof of Theorem \ref{main}}\label{mainproof}

For clarity of exposition we separate Theorem \ref{main} into three parts, considering conditions (iii), (ii), and (i) in order.  The idea for the proof of the Theorem \ref{almostdiv} is contained in \cite{BPT}.

\begin{thms}\label{almostdiv}
Let $A$ be a unital simple separable C$^*$-algebra with nonempty tracial state space.  Suppose that for each positive $b \in A \otimes \mathcal{K}$ such that $d_\tau(b) < \infty, \ \forall \tau \in \mathrm{T}(A)$ and $\delta>0$, there is a positive $c \in A \otimes \mathcal{K}$ such that
\begin{equation}\label{halfrank}
|2 d_\tau(c) - d_\tau(b)| < \delta, \ \forall \tau \in \mathrm{T}(A).
\end{equation}
It follows that for any $f \in \mathrm{Aff}(\mathrm{T}(A))$ and $\epsilon>0$ there is positive $h \in A \otimes \mathcal{K}$ with the property that 
\[
|d_\tau(h) - f(\tau)| < \epsilon, \forall \tau \in \mathrm{T}(A).
\]
If $A$ moreover has strict comparison, then for any $f \in \mathrm{SAff}(\mathrm{T}(A))$ there is positive $h \in A \otimes \mathcal{K}$ such that $d_\tau(h) = f(\tau)$ for each $\tau \in \mathrm{T}(A)$.
\end{thms}

\begin{proof}
Let $f \in \mathrm{Aff}(\mathrm{T}(A))$; we may assume $\|f\| = 1$.  It will suffice to consider $\epsilon = 1/k$ for given $k=2^n \in \mathbb{N}$.  Use Theorem \ref{hitf} to find positive $a \in A$ such that $1 \leq \|a\| \leq 1+1/k$ and $\tau(a) = f(\tau), \ \forall \tau \in \mathrm{T}(A)$.  Consider the function 
$g(t)=\sum_{i=1}^{2k+1}  (1/2k)\, \chi_{(i/2k,\infty)}(t)$.

 Using the Borel functional calculus we have $\|a-g(a)\|\leq 1/2k$, and hence
\begin{equation}\label{happrox}
|\tau(a)-\tau(g(a))|=\left| \tau(a) - \sum_{i=1}^{2k+1} (1/2k) \tau\left( \, \chi_{(i/2k,\infty)}(a)\right) \right| \leq \frac{1}{2k} 
\end{equation}
for each $\tau \in \mathrm{T}(A)$.  For each $i$  and $\tau$ we have $d_\tau((a-i/2k)_+) = \tau(  \chi_{(i/2k,\infty)}(a))$.  It is straightforward to see that (\ref{halfrank}) implies the following statement:  for $b$ as in the statement of the Theorem, for any $l \in \mathbb{N}$ and $\delta>0$, there is a positive $c \in A \otimes \mathcal{K}$ such that
\[
|l d_\tau(c) - d_\tau(b)| < \delta, \ \forall \tau \in \mathrm{T}(A).
\]
All that is actually needed is the case when $l$ is a power of two. Apply this with $l = 2k$, $b = (a-i/2k)_+$, and $\delta < 1/(4k+2)$ to obtain positive $h_i \in A \otimes \mathcal{K}$ such that
\[
|d_\tau(h_i) - (1/2k)d_\tau((a-i/2k)_+)| < 1/(4k^2+2k), \ \forall \tau \in \mathrm{T}(A).
\]
Thus, setting $h = \oplus_{i=1}^{2k+1} h_i$, we have
\begin{eqnarray*}
|\tau(g(a))-d_\tau(h)|& =&\left|\sum_{i=1}^{2k+1} (1/2k) \tau\left( \chi_{( i/2k,\infty)}(a)\right) - d_\tau \left( \oplus_{i=1}^{2k+1} h_i \right) \right| \\
& = & \left|\sum_{i=1}^{2k+1} (1/2k)d_\tau((a-i/2k)_+) - d_\tau(h_i) \right| \\
& < & \frac{2k+1}{4k^2+2k} = \frac{1}{2k}
\end{eqnarray*}
Using (\ref{happrox}) we arrive at
\[
| f(\tau) - d_\tau(h) | = | \tau(a) - d_\tau(h) |\leq |\tau(a)-\tau(g(a))|+|\tau(g(a))-d_\tau(h)|
 < \frac{1}{2k} + \frac{1}{2k} = \frac{1}{k} = \epsilon, \ \forall \tau \in \mathrm{T}(A),
\]
as desired.

Now suppose that $A$ has strict comparison.  The final conclusion of the Theorem then follows from the proof of \cite[Theorem 2.5]{bt}, which shows how one produces an arbitrary \mbox{$f \in \mathrm{SAff}(\mathrm{T}(A))$} by taking suprema.
\end{proof}

\begin{thms}\label{surjectequiv}
Let $A$ be a unital simple separable C$^*$-algebra.  Suppose that the extreme boundary $X$ of $\mathrm{T}(A)$ is compact and nonempty, and that for each $m \in \mathbb{N}$ there is $x \in \mathrm{Cu}(A)$ with the property that
\begin{equation}\label{unitdiv}
m d_\tau(x) \leq 1 \leq (m+1)d_\tau(x), \ \forall \tau \in \mathrm{T}(A).
\end{equation}
It follows that for any $f \in \mathrm{Aff}(\mathrm{T}(A))$ and $\epsilon>0$ there is positive $h \in A \otimes \mathcal{K}$ with the property that 
\begin{equation}\label{conc}
|d_\tau(h) - f(\tau)| < \epsilon, \forall \tau \in \mathrm{T}(A).
\end{equation}
If $A$ moreover has strict comparison, then we may take $f \in \mathrm{SAff}(\mathrm{T}(A))$ and arrange that $d_\tau(h) = f(\tau)$ for each $\tau \in \mathrm{T}(A)$.
\end{thms}

\begin{proof}
Let $f$ and $\epsilon$ be given, and assume $\|f \|=1$.  We need only establish (\ref{conc}) over $X$.  We may assume that $\epsilon=1/k$ for some $k \in \mathbb{N}$.  For $i \in \{0,\ldots, 2k-1\}$, let $U_i$ be the open set $\{\tau \in X \ | \ f(\tau) > i/2k \}$.  
We then have 
\begin{equation}\label{eqref:cc1}
 \left|f(\tau) - \sum_{i=0}^{2k-1} (1/2k) \chi_{U_i}(\tau) \right| \leq \frac{1}{2k}.
\end{equation}

Let us now prove the following statement:  given $1> r,\eta, \delta >0$ and an open subset $U$ of $X$,
  there is a positive element $a$ of $A \otimes \mathcal{K}$ with the property that
  \[
r-\eta \leq d_\tau(a) \leq r, \ \forall \tau \in U \ \ \mathrm{and} \ \ d_\tau(a) \leq \delta, \ \forall \tau \in X \backslash U.
\]
Choose $m$ large enough that $r-\eta < k/(m+1) < k/m < r$ for some $k \in \mathbb{N}$. 
By assumption there is $x\in \mathrm{Cu}(A)$  such that $d_\tau(x) \in [1/(m+1),1/m]$ for all $\tau\in \mathrm{T}(A)$. It follows that  $d_\tau(rx) \in [r/(m+1),r/m]$ for all $\tau\in \mathrm{T}(A)$.
 Let $c\in (A \otimes \mathcal{K})_+$ representing $rx \in \mathrm{Cu}(A)$. Then $c$ satisfies
 satisfies
\[
r-\eta < d_\tau(c) < r, \ \forall \tau \in X.
\]
Obtaining the desired element $a$ from $c$ is a straightforward application of Lemma \ref{cutoff}.

Apply the statement proved above with $U = U_i$, $r = 1/2k$, and $\eta=\delta=1/4k^2$ to obtain positive $h_i \in A \otimes \mathcal{K}$ such that 
\[
\frac{1}{2k} -\frac{1}{4k^2} \leq d_\tau(h_i) \leq \frac{1}{2k}, \ \forall \tau \in U_i
\]
and $d_\tau(h_i) < 1/4k^2$ for each $\tau \in X \backslash U_i$.  It is then straightforward to check that 
\begin{equation}\label{eqref:cc11}
\left| \sum_{i=0}^{2k-1} (1/2k) \chi_{U_i}(\tau) - \sum_{i=0}^{2k-1} d_\tau(h_i) \right| < \frac{2k-1}{4k^2}<\frac{1}{k}. 
\end{equation}

It follows from \eqref{eqref:cc1} and \eqref{eqref:cc11} that $h:= \oplus_{i=0}^{2k-1} h_i$ has the required property, since $d_\tau(h) = \sum_{i=0}^{2k-1} d_\tau(h_i)$.  

If $A$ has strict comparison then the final conclusion of the Theorem follows once again from the proof of \cite[Theorem 2.5]{bt}.

 \end{proof}

\begin{rems}
{\rm If $A$ is a unital simple separable C$^*$-algebra with the property that for any $f \in \mathrm{SAff}(\mathrm{T}(A))$ there is positive $h \in A \otimes \mathcal{K}$ such that $f(\tau) = d_\tau(h), \ \forall \tau \in \mathrm{T}(A)$, then both (\ref{halfrank}) and (\ref{unitdiv}) hold. }
\end{rems}

\begin{thms}\label{finextreme}
Let $A$ be a unital simple separable C$^*$-algebra with strict comparison of positive elements.  Suppose further that the extreme
boundary $X$ of $\mathrm{T}(A)$ is nonempty, compact, and of finite covering dimension.  It follows that for each $f \in \mathrm{SAff}(
\mathrm{T}(A))$, there is a positive element $a \in A \otimes \mathcal{K}$ with the property that $d_\tau(a) = f(\tau)$ for each
$\tau \in \mathrm{T}(A)$.
\end{thms}

\begin{proof}

By arguing as in the proof of Theorem \ref{surjectequiv}, it will suffice to find, for each $1> r,\epsilon >0$, a positive element $a$ in some $\mathrm{M}_N(A)$
with the property that
\[
r-\epsilon < d_\tau(a) < r, \ \forall \tau \in X.
\]

Set $d:=\mathrm{dim}(X)$, and for $i \in \{0,\ldots,2d+2\}$ define
\[
r_i = r - \frac{(2d+2-i)\epsilon}{2d+2}.
\]
Fix $\tau \in X$.  Use
Lemma \ref{almostdelta2} to find,
for each $k \in \{0,1,\ldots,d\}$, a positive element $\tilde{b}_k \in \mathrm{M}_N(A)$ and an open neighbourhood $V_k$ of $\tau$ in $X$ with the
property that
\begin{equation}\label{bkvalues}
r_{2k} < d_\gamma(\tilde{b}_k) < r_{2k+1}, \ \forall \gamma \in V_k.
\end{equation}
Set $U_\tau = \cap_k V_k$, so that $\mathcal{U} := \{U_\tau\}_{\tau \in X}$ is an open cover of $X$.  By the finite-dimensionality and
compactness of $X$, there are a refinement of a finite subcover of $\mathcal{U}$, say $\mathcal{W} = \{W_1,\ldots,W_n\}$, and a map
$c:\mathcal{W} \to \{0,1,\ldots,d\}$ with the property that if $i \neq j$ then
\[
c(W_i) = c(W_j) \Rightarrow W_i \cap W_j = \emptyset.
\]
Each $W_i$ is contained in some $U_\tau$, and so (\ref{bkvalues}) furnishes positive elements $\tilde{b}_k$, $k \in \{0,1,\ldots,d\}$
such that
\[
r_{2k} < d_\gamma(\tilde{b}_k) < r_{2k+1}, \ \forall \gamma \in W_i.
\]
Set $\eta = \epsilon/(n(2d+2))$, and use Lemma \ref{cutoff} to produce positive elements $b_k^{(i)}$ in $A \otimes \mathcal{K}$ with the
following properties:
\[
r_{2k} < d_\gamma(b_k^{(i)}) < r_{2k+1}, \ \forall \gamma \in W_i \ \ \mathrm{and} \ \
d_\gamma < \eta, \ \forall \gamma \in X \backslash W_i.
\]
Now for each $k \in \{0,1,\ldots,d\}$ define
\[
b_k = \bigoplus_{ \{i \ | \ c(W_i)=k\}} b_k^{(i)} \in A \otimes \mathcal{K}.
\]
The $W_i$s appearing in the sum above are mutually disjoint.  Suppose that $\tau \in W_s$
and $c(W_s) = k$.  We have the following bounds:
\begin{eqnarray*}
d_\tau(b_k) & = & \sum_{\{i \ | \ c(W_i)=k\}} d_\tau(b_k^{(i)}) \\
& = & d_\tau(b_k^{(s)}) + \sum_{ \{i \ | \ c(W_i)=k, \ i \neq s\}} d_\tau(b_k^{(i)}) \\
& < & r_{2k+1} + n\eta \\
& = & r_{2k+1} + \epsilon/(2d+2) \\
& = & r_{2k+2}
\end{eqnarray*}
and
\[
d_\tau(b_k) > d_\tau(b_k^{(s)}) > r_{2k}.
\]
For each $k \in \{0,1,\ldots,d\}$ we define
\[
\mathcal{W}_k = \bigcup_{\{i \ | \ c(W_i)=k\}} W_i,
\]
so that
\[
r_{2k} < d_\tau(b_k) < r_{2k+2}, \ \forall \tau \in \mathcal{W}_k.
\]
Note that $\mathcal{W}_0,\mathcal{W}_1,\ldots,\mathcal{W}_d$ is a cover of $X$.

To complete the proof of the Theorem we proceed by induction.  First observe that since $X$ is compact and metrizable, we may find a closed
subset $K_0$ of $\mathcal{W}_0$ with the property that $K_0^\circ,\mathcal{W}_1,\ldots,\mathcal{W}_d$ is a cover of $X$.  Set
$c_0 = b_0$, so that
\[
r-\epsilon = r_0 < d_\tau(c_0) < r_{2}, \ \forall \tau \in K_0.
\]
Now suppose that we have found a closed set $K_k \subseteq \mathcal{W}_0 \cup \cdots \cup \mathcal{W}_k$, $k < d$, such that
$K_k^\circ,\mathcal{W}_{k+1},\ldots,\mathcal{W}_d$ covers $X$, and a positive element $c_k$ in some $\mathrm{M}_N(A)$ with
the property that
\[
r-\epsilon = r_0 < d_\tau(c_k) < r_{2k+2}, \ \forall \tau \in K_k.
\]
Since $X$ is compact and metrizable, we can find a closed set $K_{k+1} \subseteq K_k^\circ \cup \mathcal{W}_{k+1}$ such that $K_{k+1}^\circ,
\mathcal{W}_{k+2},\ldots,\mathcal{W}_d$ covers $X$.  Applying Lemma \ref{almostmax} to $c_k$ and $b_{k+1}$ we obtain a
positive element $c_{k+1}$ in some $\mathrm{M}_N(A)$ with the property that
\[
r-\epsilon = r_0 < d_\tau(c_{k+1}) < r_{2k+4}, \ \forall \tau \in K_{k+1}.
\]
Starting with the base case $k=0$, applying the inductive step above $n$ times, and noting that we must have $K_d=X$, we
arrive at a positive element $c_n$ in some $\mathrm{M}_N(A)$ with the property that
\[
r-\epsilon = r_0 < d_\tau(c_n) < r_{2d+2} = r, \ \forall \tau \in X.
\]
Setting $a = c_n$ completes the proof.

\end{proof}

Theorems \ref{almostdiv}, \ref{surjectequiv}, and \ref{finextreme} together constitute Theorem \ref{main}.

  \section{Applications}\label{apps}
  
  \subsection{The structure of the Cuntz semigroup}\label{Cuntzstructure}  Let $A$ be a unital simple C$^*$-algebra with nonempty tracial state space.  Consider the disjoint union
  \[
V(A) \sqcup \mathrm{SAff}(\mathrm{T}(A)),
 \]
 where $V(A)$ denotes the semigroup of Murray-von\ Neumann equivalence classes of projections in $A \otimes \mathcal{K}$.  Equip this set with an addition operation as follows:
\vspace{2mm}
\begin{enumerate}
\item[(i)] if $x,y \in V(A)$, then their sum is the usual sum in $V(A)$;
\item[(ii)] if $x,y \in \mathrm{SAff}(\mathrm{T}(A))$, then their sum
is the usual (pointwise) sum in $\mathrm{SAff}(\mathrm{T}(A))$;
\item[(iii)] if $x \in V(A)$ and $y \in \mathrm{SAff}(\mathrm{T}(A))$,
then their sum is the usual (pointwise) sum of $\hat{x}$ and $y$ in
$\mathrm{SAff}(\mathrm{T}(A))$, where $\hat{x}(\tau) = \tau(x)$,
$\forall \tau \in \mathrm{T}(A)$.
\end{enumerate}
\vspace{2mm}
Equip $V(A) \sqcup \mathrm{SAff}(\mathrm{T}(A))$ with the partial order $\leq$ which restricts to the
usual partial order on each of $V(A)$ and $\mathrm{SAff}(\mathrm{T}(A))$,
and which satisfies the following conditions for $x \in V(A)$ and $y \in
\mathrm{SAff}(\mathrm{T}(A))$:
\vspace{2mm}
\begin{enumerate}
\item[(i)] $x \leq y$ if and only if $\hat{x}(\tau) < y(\tau)$, $\forall \tau \in \mathrm{T}(A)$;
\item[(ii)] $y \leq x$ if and only if $y(\tau) \leq \hat{x}(\tau)$, $\forall \tau \in \mathrm{T}(A)$.
\end{enumerate}
\vspace{2mm}
  
It is shown in \cite{bt} that the Cuntz semigroup of $A$ is order isomorphic to the ordered Abelian semigroup $V(A) \sqcup
\mathrm{SAff}(\mathrm{T}(A))$ defined above whenever $\iota(A \otimes \mathcal{K}) = \mathrm{SAff}(\mathrm{T}(A))$.  This structure theorem therefore applies to the algebras covered by Theorem \ref{main} provided that they have strict comparison.

\subsection{Two conjectures of Blackadar-Handelman}  In their 1982 study of dimension functions on unital tracial C$^*$-algebras---equivalently, additive, unital, and order preserving maps from the Cuntz semigroup into $\mathbb{R}^+$---Blackadar and Handelman made two conjectures (\cite{bh}):
\begin{enumerate}
\item[(i)] The space of lower semicontinuous dimension functions---dimension functions of the form $d_\tau$ for a normalized 2-quasitrace $\tau$---is weakly dense among all dimension functions.
\item[(ii)] The affine space of all dimension functions is a Choquet simplex.
\end{enumerate} 
It was proved in \cite{BPT} that these conjectures hold for a unital simple separable exact C$^*$-algebra whose Cuntz semigroup has the form described in Subsection \ref{Cuntzstructure}, and so the conjectures hold for the algebras covered by Theorem \ref{main} provided that they are exact and have strict comparison.  

\subsection{$\mathcal{Z}$-stability and a stably finite Geneva Theorem}  
At the ICM Satellite Meeting on Operator Algebras in 1994, Kirchberg announced that a simple separable nuclear C$^*$-algebra is purely infinite if and only if it absorbs the Cuntz algebra $\mathcal{O}_{\infty}$ tensorially.  This result is a cornerstone of the theory of purely infinite C$^*$-algebras and their classification. Using results of R\o rdam from \cite{R4} and the definition of strict comparison, one can rephrase Kirchberg's result:
\begin{thms}\label{zstabstrict}
Let $A$ be a simple separable nuclear traceless C$^*$-algebra.  It follows that
\[
A \cong A \otimes \mathcal{Z} \Leftrightarrow A \mathrm{ \ has \ strict \ comparison}.
\]
\end{thms} 
\noindent
Winter and the second named author have conjectured that Theorem \ref{zstabstrict} continues to hold in the absence of the "traceless" hypothesis, giving a stably finite version of the Geneva Theorem.  As in the purely infinite case, confirmations of this conjecture lead to strong classification results for simple C$^*$-algebras.  Indeed, we shall give such an application in Subsection \ref{class} below.  An important step toward the solution of this conjecture has recently been taken by Winter.
 \begin{thms}[Winter, \cite{Wi4}]\label{locfin}
 Let $A$ be a unital simple separable C$^*$-algebra with locally finite decomposition rank.  If $\mathrm{Cu}(A) \cong \mathrm{Cu}(A \otimes \mathcal{Z})$, then $A \cong A \otimes \mathcal{Z}$.
 \end{thms}
\noindent
If $A$ is a unital simple exact C$^*$-algebra with nonempty tracial state space, then the statement "The Cuntz semigroup of $A$ has the form described in Subsection \ref{Cuntzstructure}" can be neatly summarized by saying that
$\mathrm{Cu}(A) \cong \mathrm{Cu}(A \otimes \mathcal{Z})$.   We therefore have:
\begin{cors}\label{conjconfirm}
Let $A$ be a unital simple separable C$^*$-algebra with locally finite decomposition rank and nonempty tracial state space.  Suppose that $A$ satisfies any of conditions (i)-(iii) in Theorem \ref{main}.  It follows that
\[
A \cong A \otimes \mathcal{Z} \Leftrightarrow A \mathrm{ \ has \ strict \ comparison}.
\]
\end{cors}
\noindent
This represents a substantial confirmation of the conjecture described above.  We comment that locally finite decomposition rank is quite a weak property, satisfied, for instance, by any unital separable ASH algebra (\cite{NW}).  

\subsection{Classification of C$^*$-algebras}\label{class}
Winter and, Lin and Niu, have proved strong classification theorems under the assumption of $\mathcal{Z}$-stability (\cite{Wi2}, \cite{Wi3}, \cite{lin-niu}).  In light of these, we have the following confirmation of Elliott's classification conjecture.
\begin{cors}
Let $\mathcal{C}$ denote the class of C$^*$-algebras which satisfy all of the conditions of \ref{conjconfirm} and the UCT and have enough projections to separate their traces.  It follows that Elliott's conjecture holds for $\mathcal{C}$:  if $A,B \in \mathcal{C}$ and 
 \[
 \phi:\mathrm{K}_*(A) \to \mathrm{K}_*(B)
 \]
 is a graded order isomorphism with $\phi_*[1_A]=[1_B]$, then there is a $*$-isomorphism $\Phi:A \to B$ which induces $\phi$.
\end{cors}

\subsection{Classification of Hilbert modules}  
Let $A$ be a unital separable C$^*$-algebra with nonempty tracial state space, and let $E$ be a countably generated Hilbert module over $A$.  It follows from \cite{Ped:factor} that $E \cong \overline{a(A \otimes \mathcal{K})}$ for some positive $a \in A \otimes \mathcal{K}$.  If $A$ has stable rank one, then the isomorphism class of $E$ depends only on the Cuntz equivalence class of $a$, and we may define $d_\tau(E) = d_\tau(\tilde{a})$ for any $\tilde{a}$ is this equivalence class.  By Kasparov's stabilization theorem, there is a projection $P_E \in B(H_A)$ such that 
$E$ is isomorphic to $P_E H_A$. (Here $H_A = \ell^2 \otimes A$ 
is the standard Hilbert module over $A$.) 
Corollary \ref{conjconfirm} yields $\mathcal{Z}$-stability for an algebra $A$ as in Theorem \ref{main} provided that it has strict comparison and locally finite decomposition rank.  Appealing to \cite{R4}, this gives stable rank one, and so a further appeal to \cite{bt} gives the following classification result.
\begin{cors}
Let $A$ be a unital simple separable C$^*$-algebra satisfying all of the conditions of Corollary \ref{conjconfirm}.  Given two countably generated Hilbert modules $E$, $F$ over $A$, the following are equivalent: 
\begin{enumerate} 
\item[(i)] $E$ is isomorphic to $F$; 

\item[(ii)] Either $\langle P_E\rangle = \langle P_F\rangle \in \mathrm{V}(A)$ (in the case 
$P_E, P_F \in A\otimes \mathcal{K}$), or $d_\tau(E) = d_\tau(F), \ \forall \tau \in \mathrm{T}(A)$. 
\end{enumerate} 
In particular, if neither $E$ nor $F$ is a finitely generated projective module, then $E \cong F$ 
if and only if $d_\tau(E) = d_\tau(F), \ \forall \tau \in \mathrm{T}(A)$.

\end{cors}

\end{document}